\documentclass[a4paper, 11pt, leqno]{amsart}

\usepackage[T1]{fontenc}
\usepackage[utf8]{inputenc}
\usepackage{lmodern}
\usepackage{microtype}

\usepackage{geometry}

\usepackage{ifdraft}
\ifdraft{\geometry{layout=a4paper, layouthoffset=-17mm, layoutvoffset=-25mm, paper=b5paper}
}{}

\newcommand{\cprime}{\ensuremath{'}}
\usepackage{amssymb}
\usepackage{amsmath}
\usepackage{mathrsfs}
\usepackage{esint}
\usepackage{amsmath}
\usepackage[final]{hyperref}
\usepackage{enumerate}

\usepackage[mathcal]{euscript}
\usepackage{setspace}
\usepackage[all]{xy}

\usepackage{babel}
\usepackage{braket}
\usepackage{bbold}
\usepackage{dsfont}

\usepackage{enumitem}

\newcommand{\E}{\mathcal{E}}

\newcommand{\guillemet}[1]{``#1''}
\newcommand{\one}{\mathbb{1}}
\renewcommand{\set}[1]{\{#1\}}
\DeclareMathOperator{\diam}{diam}
\DeclareMathOperator{\dist}{dist}
\newcommand{\Hau}{\mathcal{H}} %Hausdorff
\newcommand{\Leb}{\mathcal{L}} %Lebesgue
\renewcommand{\d}{\,\mathrm{d}} %d
\newcommand{\D}{\mathrm{D}}
\newcommand{\dotW}{\smash{\dot{\mathrm{W}}}\vphantom{W}}

\newcommand{\R}{\mathbb{R}}
\DeclareMathOperator{\tr}{tr}
\DeclareMathOperator{\Tr}{tr}
\newcommand{\BV}{\mathrm{BV}}

\newcommand{\B}{\mathbb B}

\newtheorem{theorem}{Theorem}[section]

\newtheorem{corollary}[theorem]{Corollary}

\newtheorem{lemma}[theorem]{Lemma}
\newtheorem{proposition}[theorem]{Proposition}
\makeatletter
\newtheoremstyle{italicRemark}
{}
{}
{\normalfont}
{}
{\itshape}
{.}
{ }
{}
\makeatother

\theoremstyle{italicRemark}

\newtheorem{remark}[theorem]{Remark}

\newcommand{\pointN}{b}

\def\XXint#1#2#3{\setbox0=\hbox{$#1{#2#3}{\int}$}\vcenter{\hbox{$#2#3$}}\kern-.5\wd0}

\newenvironment{Proof}[3][\unskip]{\begin{proof}[Proof of #2 \ref{#3}{#1}]
		}{\end{proof}}

\makeatletter %

\@addtoreset{equation}{section}
\makeatother

\title[Extensions at the endpoint $p=1$ of traces for Sobolev mappings into manifolds]{Extensions of traces for Sobolev mappings into \\manifolds at the endpoint $p=1$}% the [title] provide a title with no collapse of the parameter p = 1 and the current page number

\author{Jean Van Schaftingen}
\address[J. Van Schaftingen]{
Universit\'e catholique de Louvain, Institut de Recherche en Math\'ematique et Physique, Chemin du Cyclotron 2 bte L7.01.01, 1348 Louvain-la-Neuve, Belgium}

\email{Jean.VanSchaftingen@UCLouvain.be}

\author{Benoît Van Vaerenbergh}
\address[B. Van Vaerenbergh]{
Universit\'e catholique de Louvain, Institut de Recherche en Math\'ematique et Physique, Chemin du Cyclotron 2 bte L7.01.01, 1348 Louvain-la-Neuve, Belgium}

\email{Benoit.VanVaerenbergh@UCLouvain.be}



\date{February 24, 2025}

\thanks{J. Van Schaftingen was supported by the Projet de Recherche T.0229.21 ``Singular Harmonic Maps and Asymptotics of Ginzburg--Landau Relaxations'' of the Fonds de la Recherche Scientifique--FNRS.
B. Van Vaerenbergh was supported by a FRIA fellowship}
\hypersetup{bookmarksdepth=2}

\begin{document}

\begin{abstract}
We give direct proofs and constructions of the trace and extension theorems for Sobolev mappings in $\mathrm W^{1, 1} (M, N)$, where $M$ is Riemannian manifold with compact boundary $\partial M$ and $N$ is a complete Riemannian manifold.
The analysis is also applicable to halfspaces and strips.
The extension is based on a tiling the domain of the considered applications by suitably chosen dyadic cubes  to construct the desired extension.
	Along the way, we obtain asymptotic characterizations of the $\mathrm L^1$-energy of mappings.
\end{abstract}

\maketitle

\tableofcontents

\section{Introduction and main results}\label{sec:introduction-and-main-results}
We consider two Riemannian manifolds $M$ and $N$, such that
$M$ has a compact boundary $\partial M$. We assume that $N$ is isometrically embedded in $\R^\nu$, $\nu \in \mathbb N$.
The trace operator \[\tr_{\partial M}: \dotW^{1,p}(M,N) \to \mathrm{L}^1_{\mathrm{loc}}(\partial M,N)\] extends the restriction of continuous maps $\mathrm{C}(M,N)\to \mathrm{C}(\partial M,N)$ to mappings belonging to the homogeneous Sobolev space \[
	\dotW^{1,p}(M,N) = \{ U: M \to N \text{ is weakly differentiable and } |\D U| \in \mathrm{L}^p(M)\}.
\]
Here and after, $|\D U|$ is the pointwise Frobenius or Hilbert-Schmidt norm of the weak differential $\D U$ of $U$.

For \(p > 1\), the range of the trace of \(\dotW^{1,p}(M, \R^\nu)\) is known to be the fractional space
\( \dotW^{1-1/p,p}(\partial M, \R^\nu)\) since \textsc{Gagliardo}’s seminal work \cite{Gagliardo1957caratterizzazioni};
when \(N\) is a compact manifold, the cases where the range is the fractional space \( \dotW^{1-1/p,p}(\partial M, N)\) has been characterised in a series of works over the last decades \cite[Section 6]{hardt1987mappings}\cite{bethuel1995extension}\cite{Bethuel_2014}\cite{mironescu2021trace}\cite{vanschaftingen2024extensiontracessobolevmappings}.
For our purpose it is interesting to note that the surjectivity always hold for  \(1 < p < 2 \).

Going back to \(p = 1\),
\textsc{Gagliardo} has also proved \cite{Gagliardo1957caratterizzazioni} (see also \cite{Mironescu2025note} and \cite[Teorema 10]{anzellotti1978funzioni}) that the range of the trace of \( \dotW^{1, 1}(M, \R^\nu)\) is \(\mathrm L^1 (\partial M, \R^\nu)\).
However, even in the case $N = \R^\nu$, $\nu \in \mathbb N \setminus \{0\}$, the existence of a continuous and linear extension, \emph{i.e.} the existence of a continuous and linear right inverse to $\tr_{\R^d \times \{0\}}: \mathrm{W}^{1,1}(\R^d\times (0,1)) \to \mathrm{L}^1(\R^d)$, is prevented by a result of \textsc{Peetre} \cite{peetre1979counterexample}; see also \cite[Section 5]{Pelczynski2002Sobolev}.

\textsc{Hardt} and \textsc{Lin}’s proof \cite{hardt1987mappings} works straightforwardly for \(p = 1\) \cite[Proof of Theorem 7]{bethuel1995extension} and \cite[Footnote 2]{mironescu2021trace} and shows the surjectivity of the trace when the target \(N\) is a compact Riemannian manifold.

The starting point of this work is to provide a direct proof of the surjectivity of the trace trough a fairly explicit construction of the extension.
As a byproduct, we get a slight weakening of the hypotheses from which we remove the compactness assumption on \(N\).

\begin{theorem}\label{thm:W11extmanifold}
	Let $M$ be a   compact Riemannian manifold  with compact boundary.
	Every \(U\in \dotW^{1,1}(M,N)\) has a trace \(u = \tr_{\partial M}U\) such that
	\[
	\begin{split}
	   \iint_{\partial M \times \partial M} \dist_N (u (x), u(y))\d x \d y
	   &\le C_M \iint_{\partial M \times M} \dist_N (u (x), U(y))\d x \d y\\
	   &\le C_M' \int_{M} \vert \D U\vert,
	   \end{split}
	\]
where $C_M, C_M'>0$ only depend on $M$.
Conversely, if $u:\partial M \to N$,
then there exists a map  $U \in \dotW^{1,1}(M,N)$ satisfying $\tr_{\partial M}U = u$ and
	\begin{equation}\label{eq:traceineqmanfiodl}
	\int_{M} |\D U|
	\leq C_M''
	\iint_{\partial M \times \partial M}	\dist_N(u(x),u (y)) \d x \d y.
	\end{equation}
where $C_M''>0$ only depends on $M$, provided the right-hand side is finite.
\end{theorem}

On the half-space \(\R^d \times (0, \infty)\), the statement is a little more delicate.
Indeed, it turns out that the condition
\[
 \iint_{\R^d \times \R^d} \dist_N (u (x), u(y))\d x \d y \text{ is finite}
\]
implies that \(u\) is \emph{constant} (see Proposition~\ref{prop:tobezerogeneral}).

\begin{theorem}
\label{thm:yetanothertraceineqRd}
Every $U \in \dotW^{1,1}(\R^d \times (0,\infty),N)$ has a locally integrable  trace $u = \tr_{\R^d \times \set{0}}U : \R^d \times \set{0} \to N$ such that for every \(R > 0\)
 		\begin{equation}\label{eq:theconcusionofthetracethm}
 		\begin{split}
 			\iint\limits_{\substack{\R^d \times \R^d \\ |x - y| \leq R}} \frac{\dist_{N}(u(x), u(y))}{R^d} \d x \d y
 			&\le C_d \iint\limits_{\substack{\R^d \times (\R^d \times (0, R)) \\ |x - y| \leq R}} \frac{\dist_{N}(u(x), U(y))}{R^{d + 1}} \d x \d y\\
 			&\leq C_d'\iint_{\R^d \times (0,R)}|\D U|,
		\end{split}
 		\end{equation}
 		where $C_d, C_d' > 0$ only depend $d$.\\
Conversely, if $u : \R^d \to N$ is measurable and if
\begin{equation}
\label{eq_aen5eig5uu4thahw8Eidieth}
\liminf_{R \to +\infty} \iint\limits_{\substack{\R^d \times \R^d \\ |x - y| \leq R}} \frac{\dist_{N}(u(x), u(y))}{R^d} \d x \d y <\infty
\end{equation}
then, there exists $U \in \dotW^{1,1}(\R^d \times (0,\infty),N)$ such that $\tr_{\R^d \times \set{0}}U = u$ and
 	\[
 		\int_{\R^d \times (0,\infty)} |\D U|
 		\leq C_d''
 		 \liminf_{R \to \infty} \iint\limits_{\substack{\R^d \times \R^d \\ |x - y| \leq R}} \frac{\dist_{N}(u(x), u(y))}{R^d} \d x \d y.
<+\infty.
 	\]
 	where $C_d''>0$ only depends on \(d\).`
\end{theorem}

Let us explain how the double integral in \eqref{eq:theconcusionofthetracethm}
and its limit in \eqref{eq_aen5eig5uu4thahw8Eidieth}
relate to more usual notions of integrability. First one always has for \(R > 0\) by the triangle inequality.
 \begin{equation}
 \iint\limits_{\substack{\R^d \times \R^d \\ |x - y| \leq R}} \frac{\dist_{N}(u(x), u(y))}{\Leb^d (\B^d (0, R))} \d x \d y
 		 \leq  2\inf_{b \in N}\int_{\R^d}\dist_N(u(x),\pointN) \d x.
 \end{equation}
The reverse inequality is also true asymptotically.

 \begin{theorem}\label{thm:BBMII}
 	If $u : \R^d\to N$, there exists $b_\ast \in N$ such that
 	\begin{equation}
 	\label{eq:limsupforbIIthm}
 		\int_{\mathbb{R}^d} \dist_{N}(u(x), b_\ast) \d x
 		= \frac{1}{2} \lim_{R \to +\infty}  \iint\limits_{\substack{\R^d \times \R^d \\ |x - y| \leq R}} \frac{\dist_{N}(u(x), u(y))}{\Leb^d (\B^d (0, R))} \d x \d y.
 	\end{equation}
 \end{theorem}

 The limit \eqref{eq:limsupforbIIthm} is reminiscent to  Bourgain-Brezis-Mironescu-Maz\cprime ya-Shaposhnikova-type formulae \cite{bourgain2001another}\cite{Mazya_Shaposhnikova_2002}. The volume of the $d$-dimensional ball of radius $R$ is denoted $\Leb^d (\B^d (0, R))$ and sometimes $|\B^d_R|$.

Our methods of proof also yield an extension result on strips which is the manifold constrained counterpart of \textsc{Leoni} and \textsc{Tice}'s \cite[Theorem 1.8]{leoni2019traces}.

\begin{theorem}\label{thm:w11case}
Every $U \in \dotW^{1,1}(\R^d \times (0,1),N)$ has traces
then $u_0 = \tr_{\R^d \times \{0\}}U$ and $u_1 = \tr_{\R^d \times \{1\}} U$ satisfying
\begin{equation}
\begin{split}
&\int_{\R^d} \dist_N (u_0 (x), u_1 (x)) \d x
+
\sum_{i \in \{0, 1\}} \iint_{\substack{\R^d \times (\R^d \times (0, 1))\\ |x - y| \le 1}}\dist_N (u_i (x), u_i (y)) \d x \d y\\
&\qquad \le C_d
\sum_{i \in \{0, 1\}}
\iint_{\substack{\R^d \times \R^d\\ |x - y| \le 1}}\dist_N (u_i (x), U (y)) \d x \d y\\
&\qquad \le C_d'\iint_{\R^d \times (0,1)}|\D U|.
\end{split}
\end{equation}
where $C_d, C_d' > 0$ only depend $d$.\\
Conversely, given measurable $u_0,u_1: \R^d \to N$,
there exists $U \in \dotW^{1,1}(\R^d \times (0,1), N)$ such that  $\Tr_{\R^d \times \{0\}}U = u_0$, $\Tr_{\R^d \times \{1\}}U = u_1$ and
	\begin{multline}
		\label{eq:w11intro}\int_{\R^d \times (0,1)} |\D U|\leq C_d''\int_{\R^d} \dist_N(u_0(x),u_1(x)) \d x\\ +C_d''\sum_{i \in \set{0,1}} \iint_{\substack{\R^d \times \R^d \\|x - y|\leq 1}} \dist_N(u_i(x),u_i(y)) \d x \d y
	\end{multline}
	where $C_d'' >0$ only depends on $d$, provided the right-hand side is finite.
\end{theorem}

 \emph{Plan of the paper}. 
 In Section \ref{sec:precise-assumptions-and-vocabulary-in-use}, we explain the precise assumptions we make on $N$ throughout the paper. We also explain the slightly weaker assumptions we will work with in the extension results, see \eqref{eq:thelimitassumption}, and prove
 Theorem \ref{thm:BBMII}.
 In Section \ref{sec:trace-inequalities} we write the proof of Theorem  \ref{thm:yetanothertraceineqRd} as well as a variant with manifold domain.
 The proof of Theorem \ref{thm:w11case} consists in two main steps: in Section \ref{sec:joining-two-maps-by-a-map-of-bounded-variation} we extend by a map which is constant on a countable union of rectangles that is only of bounded variation, \emph{i.e.} $\BV(\R^d \times (-1,1),N)$ and in Section \ref{sec:connecting-two-maps-by-a-sobolev-map-and-corollaries} we smoothen the BV-extension to get a $\dotW^{1,1}$ map and therefore prove Theorem \ref{thm:w11case}.
 Theorem \ref{thm:W11extmanifold} is then proven in Section \ref{sec:extension-from-manifold-domains} using a variant of Theorem \ref{thm:w11case} on cubes.

\section{Manifold-valued integrable mappings}\label{sec:precise-assumptions-and-vocabulary-in-use}

In this section, we write the precise assumptions on $N$, we define integrability, prove Theorem \ref{thm:BBMII} in the form of Proposition \ref{prop:ballBBMII}).

	\emph{The manifold $N$.} We assume that $N$ is a  (path-)connected complete manifold endowed with a distance function $\dist_N: N \times N \to [0,\infty)$ on it. The manifold is assumed to be a submanifold $N \subset \R^\nu,\nu \in \mathbb N\setminus \set{0}$ isometrically embedded by Nash isometric embedding theorem \cite{nash_imbedding_1956}; by \cite{muller_note_2009}, even if the manifold $N$ is not assumed to be compact, we may assume that $N$ is a closed set of $ \R^\nu$.
	The completeness assumption ensures that any two points can be connected by a geodesic of minimal length. This result is known as the Hopf-Rinow theorem \cite[Theorem 6.19]{lee2018introduction}. We also assume separability of $N$ : there exists a countable dense set in $N$ for the metric $\dist_N$; this assumption is of use in Proposition \ref{prop_lebesgue} to ensure that the essential range of a measurable map is separable.
	As said in Section \ref{sec:introduction-and-main-results}, $M$ is a Riemannian manifold that carry a compact boundary $\partial M$.

 A measurable map $u : \R^d \to N$ is said \emph{integrable} whenever
 \begin{equation}\label{eq:integrabili}
 	\inf_{b \in N}
 	\int_{\R^d}\dist_N(u(x),b) \d x <
 	\infty.
 \end{equation}

When $N = \R^\nu$ and $\dist_N(p,q) = |p - q|$ for all $p,q \in N$, we recover classical integrable functions augmented by their translation in $\R^\nu$.

It follows from the next proposition that the point \(b\) appearing in the condition \eqref{eq:integrabili} is unique.
 \begin{proposition}
 \label{proposition_b_unique}
 Let $u : \R^d \to N$ and $v : \R^d \to N$ be measurable, and let
 \(b, c \in N\).
 If
 \begin{equation*}
  \int_{\R^d} \!\dist_N(u(x),b) \d x, 
  \int_{\R^d} \!\dist_N(v(x),c) \d x
  \text{ and } 
  \int_{\R^d} \!\dist_N (u(x), v(x)) \d x
 \end{equation*}
 are all finite
then \(b = c\).
 \end{proposition}
\begin{Proof}{Proposition}{proposition_b_unique}
By the triangle inequality,
\begin{equation*}
\begin{split}
\int_{\R^d} \!\dist_N (b, c) \d x
&\le \int_{\R^d} \!\dist_N(b, u(x)) \d x
+
\int_{\R^d} \!\dist_N (u(x), v(x)) \d x
+
\int_{\R^d} \!\dist_N(v (x), c) \d x
\end{split}
\end{equation*}
 is finite and the conclusion follows.
\end{Proof}

 \begin{proposition} \label{prop:tobezerogeneral}
 If \(u: \R^d \to N\) is measurable and if
 	\[
 	\iint_{\R^d \times \R^d} \dist_N(u(x),u(y)) \d x\d y < \infty,
 	\]
 	then $u$ is constant.
 \end{proposition}

\begin{Proof}{Proposition}{prop:tobezerogeneral}
For every \(n \in \mathbb{N}_*\), by Fubini's theorem and comparison, there exists some point \(y_n \in \R^d\) such that
\begin{equation}
  \int_{\R^d} \dist_N(u(x),u(y_n)) \d x
  \le \frac{1}{n}.
\end{equation}
By Proposition~\ref{proposition_b_unique}, for every \(m \in \mathbb{N}_*\), \(u (y_n) = u (y_m)\). Setting \(b = u (y_n)\), we have
\begin{equation}
  \int_{\R^d} \dist_N(b,u(x)) \d x = 0,
\end{equation}
and thus \(u = b\) almost everywhere on \(\R^d\).
 \end{Proof}

The integrability can in fact be recovered from a variant of the Bourgain-Brezis-Minorescu-Maz'ya-Shaposhnikova formula \cite{bourgain2001another,brezis2002how,Mazya_Shaposhnikova_2002} (see also \cite{Gu_Yung_2021,Dominguez_Milman_2023}).

\begin{proposition}\label{prop:ballBBMII}
	If $u : \R^d\to N$ is measurable, then for every \(R > 0\) and \(b \in N\),
	\begin{equation}\label{eq:supBBMII}
		\iint\limits_{\substack{\R^d \times \R^d\\ |x - y| \leq R}} \frac{\dist_{N}(u(x), u(y))}{\Leb^d(\B^d(0,R))}\d x \d y \\\leq 2 \inf_{b \in N}\int_{\R^d} \dist_{N}(u(x),\pointN) \d x.
	\end{equation}
	and there exists $b_* \in N$ such that
	\begin{equation}\label{eq:estimateforb}
		\int_{\mathbb{R}^d} \dist_{N}(u(x), b_*) \d x =
		\frac{1}{2} \lim_{R \to \infty}  \iint\limits_{\substack{\R^d \times \R^d \\ |x - y| \leq R}} {\frac{\dist_{N}(u(x), u(y))}{\Leb^d(\B^d(0,R))} \d x \d y}.
	\end{equation}
\end{proposition}
This proves Theorem \ref{thm:BBMII}.
\begin{Proof}{Proposition}{prop:ballBBMII} We write $\B^d_R = \B^d(0,R)$.
We have, by the triangle inequality,
	\begin{equation}
	\label{eq:supBBM}
	\begin{split}
		&\iint_{\substack{\R^d \times \R^d\\ |x - y| \leq R}} \dist_N(u(x),u(y))\d x \d y \\
		&\qquad \leq \iint_{\substack{\R^d \times \R^d\\ |x - y| \leq R}}\dist_{N}(u(x),\pointN) \d x \d y + \iint_{\substack{\R^d \times \R^d\\ |x - y| \leq R}}\dist_{N}(u(y),\pointN) \d x \d y\\
		&\qquad = 2\Leb(\B^d_R) \int_{\R^d}\dist_{N}(u(x),\pointN) \d x
	\end{split}
	\end{equation}
	which proves \eqref{eq:supBBMII}.

 Set for \(R > 0\),
\[
\Theta(R)
\doteq
\iint\limits_{\substack{\R^d \times \R^d \\ |x - y| \leq R}} \frac{\dist_{N}(u(x), u(y))}{\Leb^d(\B^d(0,R))} \d x \d y.
\]
We assume without loss of generality that we have an increasing positive sequence \((R_n)_{n \in \mathbb{N}}\) such that
\[
\lim_{n \to \infty}
  \Theta(R_n)
  = \liminf_{R \to \infty}
  \Theta(R) < \infty
\]
and we set
\[
  \eta_n = \frac{1}{\max(2, R_n^{1/2})}.
\]
We have, since \(\eta_n \le 1/2\),
\[
 (\B^d_{(1 - \eta_n) R_n} \times \B^d_{\eta_n R_n})
 \cup  (\B^d_{\eta_n R_n} \times \B^d_{(1 - \eta_n) R_n})
 \subseteq \{(x, y) \in \R^d \times \R^d : \lvert x - y \rvert \le R_n\}
\]
and 
\[
 (\B^d_{(1 - \eta_n) R_n} \times \B^d_{\eta_n R_n})
 \cap  (\B^d_{\eta_n R_n} \times \B^d_{(1 - \eta_n) R_n})
 = \B^d_{\eta_n R_n}  \times \B^d_{\eta_n R_n},
\]
so that, by symmetry,
\[
\begin{split}
  &2 \int_{\B^d_{(1 - \eta_n) R_n}} \int_{\B^d_{\eta_n R_n}}
  \dist_{N} (u(x), u(y))\d x \d y\\
  &\qquad \le \int_{\B^d_{\eta_n R_n}} \int_{\B^d_{\eta_n R_n}}
  \dist_{N} (u(x), u(y))\d x \d y
  + \iint\limits_{\substack{\R^d \times \R^d \\ |x - y| \leq R_n}} \dist_{N}(u(x), u(y)) \d x \d y
\end{split}
\]
and thus
\[
\begin{split}
  &\frac{2}{\Leb^d (\B^d_{(1 - \eta_n) R_n})} \int_{\B^d_{(1 - \eta_n) R_n}} \int_{\B^d_{\eta_n R_n}}
  \dist_{N} (u(x), u(y))\d x \d y\\
  &\qquad \le \theta_n \doteq \frac{\Theta (R_n)
  + (2\eta_n)^d \Theta (2 \eta_n R_n)}{(1 - \eta_n)^d}.
  \end{split}
\]
Hence there exists \(y_n \in \B^d_{(1 - \eta_n) R_n}\) such that
\begin{equation}
\label{eq_Ooh9quiulaicheiL4eepail3}
\int_{\B^d_{\eta_n R_n}}
  \dist_{N} (u(x), u(y_n))\d x
  \le \frac{\theta_n}{2}.
  \end{equation}
If \(n\le m \), we have \(\eta_n R_n \le \eta_m R_m\).
Hence by the triangle inequality and by \eqref{eq_Ooh9quiulaicheiL4eepail3},
\[
\begin{split}
 \dist_{N} (u (y_n), u (y_m))
 &\le
 \int_{\B^d_{\eta_n R_n}} \frac{\dist_{N} (u (y_n), u (y_m))}{\Leb^d (\B^d_{\eta_n R_n})}\d x\\
 &\le
 \int_{\B^d_{\eta_n R_n}}
  \frac{\dist_{N} (u(x), u(y_n))}{\Leb^d (\B^d_{\eta_n R_n})}
  \d x +
  \int_{\B^d_{\eta_m R_m}}
  \frac
  {\dist_{N} (u(x), u(y_m))}
  {\Leb^d (\B^d_{\eta_n R_n})}\d x
  \\
 &\le \frac{\theta_n + \theta_m}{2 (\eta_n R_n)^d \Leb^d (\B^d_1)}.
\end{split}
\]
Since \(\eta_n R_n = \min(R_n/2, R_n^{1/2}) \to \infty\), the sequence
\((u(y_n))_{n}\) is a Cauchy sequence that converges to some \(b\) by completeness of \(N\). By Fatou's lemma and \eqref{eq_Ooh9quiulaicheiL4eepail3}
	\[
	\int_{\mathbb{R}^d} \dist_{N}(u(x), b) \, \d x
	\leq \lim_{n \to \infty} \int_{\B^d_{R_n}} \dist_{N}(u(x), u(y_n)) \, \d x
	\leq \lim_{n \to \infty} \frac{\theta_n}{2} = \lim_{n \to \infty} \frac{\Theta(R_n)}{2},
	\]
	which, combined with \eqref{eq:supBBMII}, gives \eqref{eq:estimateforb}.
\end{Proof}

\begin{remark}
It follows from Proposition~\ref{prop:ballBBMII} that
if $A \subset \R^d$ is measurable, then
	\[
	\Leb^d(A)
	= \frac{1}{2}
	\lim_{R \to +\infty}\frac{\Leb^{2d}(\{(x,y) \in A \times (\R^d \setminus A) : |x - y| \leq R\})}{\Leb^d(\B^d(0,R))}.
	\]
\end{remark}

\begin{proposition}
\label{prop_lebesgue}
If $u  : \R^d \to N$ is locally integrable, then for a.e. $x \in \R^d$ one has 
	\begin{equation}\label{eq:lebeguespoint}
		\lim_{r\searrow 0}\fint_{\B^d(x,r)}\dist_N(u(x),u(y))\d y = 0.
	\end{equation}
\end{proposition}
\begin{Proof}{Proposition}{prop_lebesgue}
Given \(b \in N\), we consider the function 
$f_b = \dist(u,b) \colon \R^d \to \R$.
Since 
\[
  \lvert f_b (x) - f_b (y) \rvert 
= |\dist_N (u(x), b) - \dist (u (y), b)\rvert \le \dist_N (u (x), u (y)),
\]
the function
$f_b$ is locally integrable. 
By the classical Lebesgue's differentiation theorem applied to $f_b$, there exists a set $E_b \subset \R^d$  such that $\Leb^d(E_b) = 0$ and such that for every $x \in \R^d\setminus E_b$,
	\begin{equation*}
		\lim_{r\searrow 0}\fint_{\B^d(x,r)}\dist_N(u(y),b)\d y = \dist_N(u(x),b), 
	\end{equation*}
	and thus 
	\begin{equation}
	\label{eq_eec4ahlai7ahpae4eehaem0G}
	\begin{split}
	\limsup_{r\searrow 0}\fint_{\B^d(x,r)}\dist_N(u(y),u(x))\d y 
	&
	\leq \lim_{r\searrow 0}\fint_{\B^d(x,r)}\dist_N(u(y),b) + \dist_N(u(x),b)\d y\\
	&= 2\dist_N(u(x),b).
	\end{split}
	\end{equation}	
	By our separability assumption on $N$, considering a countable dense set $B \subset N$, we write $E = \bigcup_{b \in B}E_b$ and observe that $\Leb^d(E) = 0$.
	Moreover, for every $b \in B$ and $x \in \R^d \setminus E$, \eqref{eq_eec4ahlai7ahpae4eehaem0G} holds and its right right-hand side can be made arbitrary small. 
	We have shown that for almost every $x \in \R^d$, \eqref{eq:lebeguespoint} holds.
\end{Proof}

\begin{proposition}
\label{proposition_Lp_to_Linfty_translation}
If $u : \R^d \to N$ is measurable,
then for every  $u : \R^d \to N$, $R > 0$ and $h \in \B^d(0,R)$, one has
\begin{equation}\label{eq:sdfghjkkjhgfds}
	\int_{\R^d}\dist_N(u(x),u(x + h))\d x \leq
	C_d \iint_{\substack{\R^d \times \R^d \\ |x - y| \leq R}}\frac{\dist_{N}(u(x),u(y))}{\Leb^d(\B^d(0,R))}\d x \d  y
\end{equation}
where $C_d>0$ only depends on $d$.
\end{proposition}

\begin{lemma}
\label{lemma_L1_to_Linfty}
If \(x, y \in \R^d\) and \(\lvert x - y \rvert \le R\),
then
\[
 \dist_N (u (x), u (y))
 \le A_d \Bigl( \fint_{\B^d (x, R)} \dist_N (u (x), u (z))\d z + \fint_{\B^d (y, R)} \dist_N (u (x), u (z))\d z  \Bigr),
\]
where \(A_d > 0\) only depends on \(d\).
\end{lemma}

\begin{Proof}{Lemma}{lemma_L1_to_Linfty}
By the triangle inequality, we have succesively
\begin{equation*}
\begin{split}
 &\dist_N (u (x), u (y))\\
 &\qquad = \fint_{\B^d (x, R) \cap \B^d (y, R)}
 \dist_N (u (x), u (y)) \d z \\
 &\qquad \le \fint_{\B^d (x, R) \cap \B^d (y, R)}
 \dist_N (u (x), u (z)) \d z  + \fint_{\B^d (x, R) \cap \B^d (y, R)}
 \dist_N (u (y), u (z)) \d z \\
 &\qquad \le A_d  \Bigl( \fint_{\B^d (x, R)} \dist_N (u (x), u (z))\d z + \fint_{\B^d (y, R)} \dist_N (u (x), u (z))\d z  \Bigr),
 \end{split}
\end{equation*}
where
\[
%\begin{split}
	A_d \doteq \sup\left\{ \frac{|\B^d(x;R)|}{|\B^d(x;R)\cap \B^d(y;R)|} : x,y \in \R^d, R > 0, |x - y| \leq R\right\}\\
	%&= \sup\left\{ \frac{|\B^d(0;R)|}{|\B^d(0;R)\cap \B^d(y;R)|} : y \in \R^d, R > 0, |y| = R\right\}
%\end{split}
\]
is finite.
\end{Proof}

\begin{Proof}{Proposition}{proposition_Lp_to_Linfty_translation}
Applying Lemma~\ref{lemma_L1_to_Linfty} and integrating over $x \in \R^d$, we obtain
\begin{equation*}
\begin{split}
		\int_{\R^d}\dist_N(u(x),u(x + h)) \d x
		&\le  \frac{A_d}{\Leb(\B^d(0; R))}
		\int_{\substack{\R^d \times \B^d\\ \vert x - z\vert \le R}}
		\dist_N (u (x), u (z))\d x \d z\\
		& \qquad
		+ \frac{A_d}{\Leb(\B^d(0; R))} \int_{\substack{\R^d \times \R^d\\
		\lvert x + h - z \rvert \le R}}
		\dist_N (u (x + h ), u (z))\d x \d z\\
		&= \frac{A_d}{\Leb(\B^d (0; R)}
		\int_{\substack{\R^d \times \B^d\\ \vert x - z\vert \le R}}
		\dist_N (u (x), u (z))\d x \d z
		\end{split}
\end{equation*}
which is \eqref{eq:sdfghjkkjhgfds}.
\end{Proof}

\begin{proposition}\label{prop:ghjklkjhgfd}
If $u : \R^d \to N$ is measurable, if for some \(R > 0\),
\begin{equation}
\label{eq_UghoShihee4AhphaeK6aeFuo}
\iint_{\substack{\R^d \times \R^d \\|x - y|\leq R}}\frac{\dist_N(u(x),u(y))}{\Leb^d (\B^d (0, R))} \d x \d y < \infty
\end{equation}
then
\begin{equation}
\label{eq_oaquohh6ohVeij4Ootaewook}
 \lim_{r\searrow 0}\iint_{\substack{\R^d \times \R^d \\|x - y|\leq r}}\frac{\dist_N(u(x),u(y))}{\Leb^d (\B^d (0, r))} \d x \d y= 0.
\end{equation}
\end{proposition}

\begin{Proof}{Proposition}{prop:ghjklkjhgfd}
Given \(R > 0\), we define the function
\[
 f (x) = \fint_{\substack{\B^d (x, R)}} \dist_N(u(x),u(z)) \d z.
\]
In view of \eqref{eq_UghoShihee4AhphaeK6aeFuo}, we have
\[
 \int_{\R^d} f < \infty
\]
whereas by Lemma~\ref{lemma_L1_to_Linfty}, for every \(x, y \in \R^d\) such that \(\vert x - y \vert \le R\),
\[
  \dist_N (u (x), u (y))
  \le A_d (f (x) + f (y)).
\]
Averaging the latter inequality over $x \in \B^d(x,r)$, we get
\begin{equation}
\label{eq_cisuyoh7AhKe0pesoosi2zoc}
 \fint_{\B^d (x, r)} \dist_N (u (x), u (y))\d y
 \le A_d f (x) + A_d \fint_{\B^d (x, r)} f.
\end{equation}
Since the left-hand side of \eqref{eq_cisuyoh7AhKe0pesoosi2zoc} converges almost-everywhere and its right-hand side converges in \(\mathrm L^1(\R^d)\) as \(r \to 0\) by approximation by averages in \(\mathrm L^1\), \eqref{eq_oaquohh6ohVeij4Ootaewook} follows from Lebesgue's dominated convergence.
\end{Proof}

\begin{remark}
	From \eqref{eq:sdfghjkkjhgfds}, we see that the finitness \eqref{eq_UghoShihee4AhphaeK6aeFuo} implies that 
	\[
	\sup_{r \in  (0,R)} \ \iint_{\substack{\R^d \times \R^d \\ |x - y| \leq r}}\frac{\dist_{N}(u(x),u(y))}{\Leb^d (\B^d (0, r))}
	\d x \d  y \leq
	C_d  \iint_{\substack{\R^d \times \R^d \\ |x - y| \leq R}}\frac{\dist_{N}(u(x),u(y))}{\Leb^d (\B^d (0, R))}
	\d x \d  y.
	\]
\end{remark}

We finally record the following easy-to-state consequence of Propositions \ref{prop:ghjklkjhgfd} and \ref{proposition_Lp_to_Linfty_translation}:
\begin{corollary}
	If $u : \R^d \to N$ is measurable, then
	\begin{equation}%\label{eq:sdfghjkkjhgfds}
		\lim_{h \to 0}\int_{\R^d}\dist_N(u(x),u(x + h))\d x
		=
		\lim_{r\searrow 0}\iint_{\substack{\R^d \times \R^d \\|x - y|\leq r}}\frac{\dist_N(u(x),u(y))}{\Leb^d (\B^d (0, r))} \d x \d y
		\in \{0,+\infty\}.
	\end{equation}
\end{corollary}

\section{Trace inequalities}\label{sec:trace-inequalities}

In this section, we write a proof of Theorem \ref{thm:yetanothertraceineqRd} as well as a variant on balls (Proposition \ref{prop:yetanothertraceineq}) and when the domain is a manifold, see Proposition \ref{prop:manifoldtrace}.
\begin{proposition}\label{prop:yetanothertraceineq}  Fix $R>0$. If $U \in \dotW^{1,1}(\B^d(0,R) \times (0,1),N)$, the trace $u = \tr_{\B^d(0,R) \{0\}}U  : \B^d(0,R) \times \set{0} \to N$ is locally integrable and  satisfies for each $r \in (0,R)$
	\begin{equation}\label{eq:theconcusionofthetracelemma}
		\iint\limits_{\substack{\B^d(0,R) \times \B^d(0,R) \\ |x - y| \leq r}} \frac{\dist_N(u(x), u(y))}{r^d}   \d x \d y \leq C_d\iint_{\B^d(0,R) \times (0,r)}|\D U|
	\end{equation}
	and
	\begin{equation}\label{eq:theconjunctionofthetracelemma}
		\iint\limits_{\substack{\B^d(0,R) \times (\B^d(0,R) \times (0, R)) \\ |x - y| \leq r}}
		\frac{\dist_N(u(x), U(y))}{r^{d + 1}}   \d x \d y \leq C_d\iint_{\B^d(0,R) \times (0,r)}|\D U|
	\end{equation}
	where $C_d > 0$ only only depends on $d$.
\end{proposition}

In the proof of Proposition \ref{prop:yetanothertraceineq}, we use the fact that if $C \subset \R^d$ is a convex and $U \in \dotW^{1,1}_{\mathrm{loc}}(C, N)$ then for almost every $x, y \in C$, 
\begin{equation}\label{eq:fundamentalthm}
	\dist_N(U(x),U(y)) \leq |x - y|\int_0^{1}|\D U((1 - t)x + ty))|\d t
\end{equation}
and the right-hand side is finite.

\begin{Proof}{Proposition}{prop:yetanothertraceineq} By scaling, we assume $R = 1$ and set $\B^d = \B^d(0,1)$ the unit ball centered at the origin.
For almost every \(x \in \B^d \simeq \B^d \times \{0\}\) and \(y \in \B^d \times (0, r)\), we have, by \eqref{eq:fundamentalthm},
\begin{equation*}
	\dist_N(u(x), U(y))\leq
	|x - y| \int_0^1|\D U((1 - t)x + t y)|\, \mathrm{d}t.
\end{equation*}
and thus, with \(y = (y', y_{d + 1})\),
	\begin{multline}
	\label{eq_EitooKoh9siedootaenoog9W}
		 \iint_{\B^d \times (\B^d \times (r/2, r))} 	\dist_N(u(x), U(y)) \one_{\{|x - y'|< r\}}  \, \mathrm{d}x\, \mathrm{d}z\\
		\leq  \int_0^1 \iint_{\B^d \times (\B^d\times (0, r))}|x - y|  \one_{\{|x - y'|< r\}}|\D U((1 - t)y + t x)| \, \mathrm{d}x\, \mathrm{d} y\, \mathrm{d}t.
	\end{multline}
Changing the variable $y$ to $z = (1 - t)x + t y$
so that the Jacobian determinant of the deformation is equal to $t^{d + 1}$, \(x - z = t(x - y)\) and \(z_{d + 1} = t y_{d + 1}\), the right-hand side of \eqref{eq_EitooKoh9siedootaenoog9W} is controlled by
	\begin{equation}
	 \label{eq:todiscussdim}
	 \begin{split}
	 &\int_0^1 \iint_{\B^d \times (\B^d \times (r/2, r))} |x - z|
		\one_{\{\max(|x - z'|, z_{d + 1}) < t r < 2 z_{d + 1}\}}|\D U(z)|\frac{\d z \d w}{t^{d + 2}}  \d t\\
&\qquad \le
		\int_{\B^d \times (r/2,r)}
		|\D U(z)| \int_{\B^{d}} \int_{\max(\vert x - z\vert, z_{d + 1}) /r}^{\infty}
		\frac{\vert x - z \vert \one_{\{|x - z'| < r\}}}{t^{d + 2}} \d t \d z\\
		&\qquad \le  \frac{r^{d + 1}}{d + 1}
		\int_{\B^d \times (r/2,r)}
		|\D U(z)| \int_{\B^{d}}
		\frac{|x - z| \one_{\{|x - z'| \leq 2z_{d + 1} \}}}{\max(\vert x - z' \vert, z_{d + 1})^{d + 1}} \d t \d z\\
		&\qquad \le 2^{d + \frac{1}{2}} \vert \B^d\vert
		\int_{\B^d \times (0,1)}
		|\D U(z)| \d z.
		\end{split}
	\end{equation}

Next, we have
\begin{equation}
\label{eq_xeizoh8Uyudo0aihoo8eo1yo}
	\begin{split}
	&\iint_{\B^d \times (\B^d \times (0, r/2))} 	\dist_N(U(y), U(y', y_{d + 1}+r/2)) \one_{\{|x - y'|< r\}}  \, \mathrm{d}x\, \mathrm{d}z\\
   &\qquad \le 	\iint_{\B^d \times (\B^d \times (0, r/2))}
   \int_0^{r/2} \vert D U (y', y_{d + 1} + t) \one_{\{|x - y'|< r\}} \vert \d t\\
   &\qquad < \frac{r^{d + 1} \vert \B^d\vert}{2}
   \int_{\B^d \times (0, r)} \vert DU \vert.
   \end{split}
\end{equation}
Combining \eqref{eq_EitooKoh9siedootaenoog9W}, \eqref{eq:todiscussdim} and \eqref{eq_xeizoh8Uyudo0aihoo8eo1yo}, we get
\[
\iint_{\B^d \times (\B^d \times (0, r))} 	\dist_N(u(x), U(y)) \one_{\{|x - y'|< r\}}  \, \mathrm{d}x\, \mathrm{d}z
\le
\bigl(2^{d + 1 + \frac{1}{2}} + \tfrac{1}{2}\bigr)
\vert \B^d \vert
\int_{\B^d \times (0, r)} \vert DU \vert,
\]
from which \eqref{eq:theconjunctionofthetracelemma} follows.

Finally, we have by the triangle inequality
\begin{equation}
\label{eq_kies1weemaigau5si2Iv4loh}
\begin{split}
 &\iint_{\B^d \times \B^d \times (\B^d \times (0, r))} 	\dist_N(u(x), u(y)) \one_{\{|x - z|< r, |y - z| < r\}}  \, \mathrm{d}x\, \mathrm{d}z\\
&\le
 \iint_{\B^d \times (\B^d \times (0, r))} 	\dist_N(u(x), U(z)) \one_{\{|x - z|<r\}}  \, \mathrm{d}x\, \mathrm{d}z\\
&\qquad \qquad  +  \iint_{\B^d \times (\B^d \times (0, r))} 	\dist_N(u(y), U(z)) \one_{\{|y - z|< r\}}  \, \mathrm{d}x\, \mathrm{d} y\\
&\le \bigl(2^{d + 2 + \frac{1}{2}} + 1\bigr)
\vert \B^d \vert
\int_{\B^d \times (0, r)} \vert \D U \vert.
 \end{split}
\end{equation}
We note that there exists $B_d>0$ depending only on $d$ such that
\begin{equation}
\label{eq_ip2aeyahPieb6Thah6zeegee}
 \int_{\B^d \times (0, r)}
 \one_{\{|x - z|< r, |y - z| < r\}}
 \d y \d z
 \ge B_d r^{d + 1} \one_{\{\vert x - y \vert < r\}}.
\end{equation}
Finally, \eqref{eq:theconcusionofthetracelemma} follows from \eqref{eq_kies1weemaigau5si2Iv4loh} and \eqref{eq_ip2aeyahPieb6Thah6zeegee}.
		\end{Proof}

We record that taking $R \to +\infty$ one obtains a similar result on $\R^d$ (instead of $\B^d(0,R)$) holding for all $r>0$; this result is written as Theorem \ref{thm:yetanothertraceineqRd}.

\begin{proposition}\label{prop:manifoldtrace} Let $M$ a Riemannian manifold with compact boundary $\partial M$. There exists $R_0 >0$ and  such that if $U \in \dotW^{1,1}(M,N)$, the trace $u = \tr_{\R^d \times \{0\}}U :\partial M \to N$ is integrable and  satisfies for each $r \in (0,R_0)$,
	\begin{equation*}
		\iint\limits_{\substack{\partial M \times \partial M \\ \dist_{\partial M}(x, y) \leq r}}
		\frac{\dist_N(u(x), u(y))}{r^{\dim(\partial M)}} \d x \d y \leq C_M \iint\limits_{M \cap \{\dist_M(x, \partial M) \leq r\}} |\mathrm{D} U|
	\end{equation*}
and
	\begin{equation*}
		 \iint\limits_{\substack{\partial M \times M \\ \dist_{\partial M}(x, y) \leq r}} \frac{\dist_N(u(x), U(y))}{r^{\dim(\partial M) + 1}} \d x \d y \leq C_M \iint\limits_{M \cap \{\dist_M(x, \partial M) \leq r\}} |\mathrm{D} U|
	\end{equation*}
	where $C_M>0$ only depends on $M$.
\end{proposition}
\begin{proof}
This is a consequence of Proposition~\ref{prop:yetanothertraceineq} on a finite subcovering of $\partial M \times \partial M \subset \bigcup_{x \in \partial M}\B^M(x,r) \times \B^M(x,r)$ by geodesic balls $\B^M(x,r)$, $x \in \partial M$ and the fact that $\B^M(x,r)$ and $\B^{\dim(\partial M)}(0,1) \times (0,1)$ are Lipschitz de\-for\-ma\-tions of each other.
\end{proof}

\section{Joining two maps by a map of bounded variation}\label{sec:joining-two-maps-by-a-map-of-bounded-variation}

In this section we provide the first step of the proof of the extensions results: Proposition \ref{prop:smalllemma_varconnect2map} joins two integrable mappings by a map of bounded variation.

Fix $L >0$ and $k \in \mathbb Z$ and let us denote by $\mathcal Q_{k} = \{Q(a,2^{-k}L): a \in 2^{1-k}L\mathbb Z^d\}$ the \emph{dyadic cubes} of radius $2^{-k}L$.  We will crucially use that $\Leb^d(Q) = (2^{1 - k}L)^d$ which we will sometimes write $|Q|$.

\begin{proposition}\label{prop:smalllemma_varconnect2map}  
	If $u_0,u_1: \R^d \to N$ satisfy for $i \in \set{0,1}$ 
	\begin{equation}\label{eq:thelimitassumption}%\label{eq:assumption_smallllemma_varconnect2map}
		\inf_{r \in \R}\iint_{\substack{\R^d \times \R^d \\|x - y|\leq r}}\frac{\dist_N(u_i(x),u_i(y))}{r^d} \d x \d y < \infty
	\end{equation} 
	and if
	\begin{equation}
	\label{eq:intecondition}
	\int_{\R^d} \dist_N(u_0(x),u_1(x)) \d x < \infty,
\end{equation}
	then,  for each $L >0$,  there exists $U  \in \dot\BV(\R^d \times (-L,L), N)$ such that  \begin{equation}\label{eq:tr_in_BV}
		\Tr_{\R^d \times \{-L\}}U = u_0 \quad \text{ and } \quad \Tr_{\R^d \times \{L\}}U = u_1.
	\end{equation}
	Moreover,  there exists an increasing sequence \( (k_{n})_{n \in \mathbb{N}} \subset \mathbb{N}\) depending on $u_i, i \in \set{0,1}$ such that $k_{0} = 0$  and such that setting
	\begin{align}
		\label{eq_ohphae9Ahhuloo5Ooka6ho7y}
		I_{n, 0} = (- (1 - 2^{-k_{n+1}})L, -(1 - 2^{-k_{n}})L) 
		\text{ and }
		I_{n,1} = ((1 - 2^{-k_{n+1}})L, (1 - 2^{-k_{n}})L),
	\end{align}
	and	we have for each \( t \in I_{n,i} \) and each \( x \in Q \in \mathcal{Q}_{k_{n}} \), \(U(x, t) = u_i(x_Q)\)	for some Lebesgue point \( x_Q \in Q \) of \( u_i \).
	Moreover, if $t \in I_{n, i}$,  \begin{equation}\label{eq:tracelimit}
				\int_{\R^d}\dist_N(u_i(x), \tr_{\R^d \times \{t\}}U(x))\d x \leq \frac{2^{-n}}{L^d}\sum_{i \in \set{0,1}} \iint_{\substack{\R^d \times \R^d \\|x - y|\leq L}}	\dist_N(u_i(x),u_i(y)) \d x \d y.
	\end{equation}
	If we write $\mathrm{J}_U$ for the union of the faces $F$ of the rectangular cuboïds $Q \times I_{n}$, $Q \in \mathcal Q_{k_{n,i}}$ and denote by $\Tr_{\mathrm{J}_U}^+U$ and $\Tr_{\mathrm{J}_U}^-U$ the value of $U$ on the two cuboïds sharing a same face $F$,  we have
	\begin{equation}
	\label{eq:ineq_smallllemma_varconnect2map}
	\begin{split}
		&\int_{\mathrm{J}_U \cap \R^d\times (-L,L)} \dist_N(\Tr_{\mathrm{J}_U}^+U(x),\Tr_{\mathrm{J}_U}^-U(x))\d x\\
		&\qquad \leq  \int_{\R^d} \dist_N(u_0(x),u_1(x)) \d x
		+\frac{C_d}{L^d}\sum_{i \in \set{0,1}} \iint_{\substack{\R^d \times \R^d \\|x - y|\leq L}}	\dist_N(u_i(x),u_i(y)) \d x \d y,
		\end{split}
	\end{equation}
	where $C_d >0$ only depends on $d$.
\end{proposition}

 The homogeneous space of manifold constrained mappings of bounded variation  $\BV(\R^d \times (-L,L), N)$ is defined through $N \subset \R^\nu$: 
 \begin{multline*}
 	\dot\BV(\R^d \times (-L,L), N) \\= \{u \in \dot\BV(\R^d \times (-L,L), \R^\nu): u(x) \in N \text{ for almost everywhere } x \in N\}.
 \end{multline*}
 We refer to \cite{ambrosio2000functions} for the definition of $\BV(\R^d \times (-L,L), \R^\nu)$; the corresponding homogeneous space $	\dot\BV(\R^d \times (-L,L), N)$ neglects the $\mathrm{L}^1$ part of the definition. One can show that $\mathrm{J}_U$ and $\Tr_{\mathrm{J}_U}^\pm$ in Proposition \ref{prop:smalllemma_varconnect2map} are respectively the jump set and the one sided traces of BV maps as defined in \cite[Theorem 3.77]{ambrosio2000functions}. The trace of mappings of bounded variations appearing in \eqref{eq:tr_in_BV} is defined in \cite{anzellotti1978funzioni}. For our purposes, one has by \eqref{eq:tracelimit}, for $i \in \set{0,1}$,
 \[
 		\lim_{n \to \infty}\sup_{t \in I_{i,n}}\int_{\R^d}\dist_N(u_i(x), \tr_{\R^d \times \{t\}}U(x))\d x = 0.
 \]
 The trace operator $\BV \to \mathrm{L}^1_{\mathrm{loc}}$ being continuous under translations, we deduce \eqref{eq:tr_in_BV}.

\begin{Proof}{Proposition}{prop:smalllemma_varconnect2map} Let $i \in \set{0,1}$. 
	For each $k \in \mathbb Z$, for each $Q \in \mathcal Q_k$, there exists a Lebesgue point $x_Q \in Q$ of $u_i$ such that
	\begin{equation}\label{eq:medianinequality}
		\fint_Q \dist_N(u_i(x_Q),u_i(y)) \d y \leq \fint_Q\fint_Q \dist_N(u_i(x),u_i(y))\d y \d x,
	\end{equation}
	where here and after for sets $\Omega$ of finite measure, we write
	\[
		\fint_{\Omega} f \d \mu  = \frac{1}{\mu(\Omega)}\int_{\Omega}f\d \mu,
	\] 
	the mean of $f : \Omega \to \R$ on $\Omega$.

	For each $k \in \mathbb Z$, we define for $x \in \R^d$,  
	\begin{align}\label{eq:choiceofuQ}
		\E_k (u_i)(x) &= \E_k(u_i)(Q) = u_i(x_Q) &
		&\text{ if }x \in Q \in \mathcal Q_k.
	\end{align}
	 We will use the second expression $\E_k(u_i)(Q) \in N$ to emphazise that $\E_k(u_i)(x)$ is constant on the cube $Q$. 
	
	By our assumption \eqref{eq:thelimitassumption} and Proposition \ref{prop:ghjklkjhgfd}
	\begin{equation}
	\label{eq:limittozero}
	\begin{split}
		&\varlimsup_{k \to +\infty}\int_{\R^d}\dist_N(\E_{k} (u_i)(x), u_i(x))  \d x 		\\
		&\qquad \leq 	\varlimsup_{k \to +\infty}\sum_{Q \in \mathcal Q_{k}}\int_Q\fint_Q \dist_N(u_i(x),u_i(y))\d y \d x\\
		&\qquad \leq \lim_{k \to \infty}\frac{1}{2^{(1-k)d}L^d}\iint_{\substack{\R^d \times \R^d \\ |x - y|_\infty \leq 2^{1 -k}L}}\dist_N(u_i(x),u_i(y)) \d x \d y= 0.
		\end{split}
	\end{equation}
	For later we set $k_{0} = 0$ and define
	\begin{equation}\label{eq:choiceofalpha}
		\Gamma \doteq 
		 \sum_{i \in \set{0,1}}\frac{1}{L^d}\iint_{\substack{\R^d \times \R^d \\|x - y|\leq L}}	\dist_N(u_i(x),u_i(y)) \d x \d y.
	\end{equation}
	In view of \eqref{eq:limittozero}, there exists, a increasing sequence of natural numbers $(k_{n})_n$ so that for each $n \in \mathbb N_*$,  $k \geq k_{n}$ and \(i \in \{0, 1\}\),
	\begin{equation}\label{eq:alphaover2ell}
		\int_{\mathbb R^d} \dist_N(\E_{k} (u_i)(x),u_i(x)) \d x \leq \frac{\Gamma}{2^{n}}.
	\end{equation}
	By  \eqref{eq:thelimitassumption} and a change of variable,
	\begin{equation}\label{eq:meanstroanslation}
		\lim_{k \to +\infty}\int_{\R^d}\fint_{|h|_\infty \leq 2^{1 - k}L}\dist_N(u_i(x), u_i(x + h)) \d h \d x  = 0.
	\end{equation}	
	By \eqref{eq:meanstroanslation}, we may further assume that there exists a subsequence $(k_{n,i})_n$ satisfying for each $n \in \mathbb N_*$
	\begin{equation}\label{eq:distbtwcubes}
		\int_{\mathbb R^d} \fint_{|h|_\infty \leq 2^{1-k_{n}}L}\dist_N(u_i(x),u_i(x +  h)) \d h \d x \leq \frac{\Gamma}{2^{k_{n-1}}}
	\end{equation}
	by mathematical induction.

We set
\[ 
	U(x,t) \doteq \E_{k_{n}} (u_i)(x) \; \text{ if } \;i\in \set{0,1},\; t \in I_{n,i},\; x \in \mathcal Q_{k_{n}},
\]
where \(I_{n, i}\) is defined in \eqref{eq_ohphae9Ahhuloo5Ooka6ho7y}.
We obtain a map \( U: \mathbb{R}^d \times (-L,L) \to N \) which is constant on rectangular cuboïds; the manifold constraint is satisfied as  $\E_{k_{n}}( u_i)(x) \in N$ by \eqref{eq:medianinequality}. The map is in fact only defined almost everywhere as no value was given on the interfaces of the cuboïds. By the choice of $\Gamma > 0$ it proves \eqref{eq:tracelimit}.
We claim that $\Tr_{\R^d \times \set{-L}}U = u_0$ and $\Tr_{\R^d \times \{L\}}U = u_1$. The situation is symmetric so we only consider the case $i =0$ in $t = R$.
The trace operator $\Tr_{\R^d \times \set{0}}: \mathrm{BV}_{\mathrm{loc}}(\R^d \times (-L,L)) \to \mathrm L^1(\R^d)$ is continuous with respect with translations. We have $\tr_{\R^d \times \{\tau\}}U = \E_{k_{n,1}}(u_0)$ if $\tau \in I_{n,0}$. By \eqref{eq:limittozero}, we deduce that $\Tr_{\R^d \times \{L\}}U = u_1$ and \eqref{eq:tracelimit}.
Therefore it only remains to prove \eqref{eq:ineq_smallllemma_varconnect2map}.

	Since $U$ is constant on rectangular cuboïds by construction and jumps on the interfaces of the rectangles, we deduce that $U \in \mathrm{BV}_{\mathrm{loc}}(\R^d \times (-L,L))$.
	The measure $|\mathrm \D U|$ is absolutely continuous with respect to $\Hau^{d}$ and we will denote its support $\mathrm{J}_U$.
	We therefore deduce that the left-hand side in the Proposition \ref{prop:smalllemma_varconnect2map}\eqref{eq:ineq_smallllemma_varconnect2map},
	is controlled by the three following contributions:
	We will estimate \ref{item:umeetsv}, \ref{item:parallelu} and \ref{item:perpu}
	separately.

	\begin{enumerate}[label=(\roman*)]
	\item \label{item:umeetsv}  Contributions at the interface $\R^d \times \set{0}$ in $t = 0$
	\begin{equation}\sum_{\substack{Q \in \mathcal Q_{k_{0}}}}\Hau^d(Q)\dist_N(\E_{k_{0}}(u_0)(Q), \E_{k_{0}}(u_1)(Q)).
	\end{equation}
	
	\item \label{item:parallelu} \emph{Parallel} contributions to $\R^d\times \{0\}$ of each $u_i$, $i \in \set{0,1}$, (except the jump on $\R^d \times \{0\}$, see \ref{item:umeetsv})
	\begin{equation}
		\sum_{i \in \set{0,1}}\sum_{n \in \mathbb N}\sum_{Q \in \mathcal Q_{k_{n}} }\sum_{\substack{Q' \in \mathcal Q_{k_{n+1}} \\ Q' \subset Q} }\Hau^d(Q') \dist_N(\E_{k_{n}} (u_i)(Q), \E_{k_{n+1}} (u_i)(Q')).\label{eq:parallelu}
	\end{equation}
	\item \label{item:perpu} \emph{Perpendicular} contributions to $\R^d\times \{0\}$
	\begin{equation}
		\sum_{i \in \set{0,1}}\sum_{n \in \mathbb N}\sum_{Q \in \mathcal Q_{k_{n}} }\sum_{\substack{Q' \in \mathcal Q_{k_{n}}\\\cap \mathrm{neigh}(Q)}}\Hau^d(\bar Q \cap \bar Q' \times I_{k_n}) \dist_N(\E_{k_{n}}( u_i)(Q'), \E_{k_{n}} (u_i)(Q)) \label{eq:perpu}
	\end{equation} 
	where the fourth sum runs  over the $2d$ neighbours  of the dyadic cube $Q \in \mathcal Q_{k_n}$.
	We record that the measure of a intersecting lateral face of two cuboïds, denoted $\bar Q \cap \bar Q'\times I_{k_n}$, is $\Hau^d(\bar Q \cap \bar Q'\times I_{k_n}) = 2^{(1 - k_{n})(d - 1)}(2^{-k_{n+1}}-{2^{-k_{n}}})L^d$.
\end{enumerate}

The analysis of \ref{item:umeetsv} relies on the triangle inequality and \eqref{eq:medianinequality}. Indeed, the sum  over cubes $\mathcal Q_{k_{0}}$ of  \ref{item:umeetsv} equals
\begin{equation}
\begin{split}
		&\int_{\R^d}\dist_N(\E_{k_{0}}(u_0)(x), \E_{k_{0}}(u_1)(x))\d x  \\
	&\qquad \leq  \int_{\R^d}\dist_N(\E_{k_{0}}(u_1) (x), u_0(x))\d x\\
	&\qquad \qquad +\int_{\R^d}\dist_N(u_0(x),u_1(x))\d x + \int_{\R^d}\dist_N(u_1(x),\E_{k_{1,0}}(u_1)(x))\d x\\
	&\qquad \leq  \int_{\R^d}\dist_N(u_0(x),u_1(x))\d x\\
	&\qquad \qquad + \sum_{i \in \set{0,1}}\sum_{Q \in \mathcal Q_{k_{0}}}|Q|^{-1}\iint_{Q\times Q}\dist_N(u_i(x),u_i(y)) \d x \d y\\
	&\qquad \leq \int_{\R^d}\dist_N(u_0(x),u_1(x))\d x\\
	&\qquad \qquad + L^{-d}\sum_{i \in \set{0,1}} 2^{d(k_{0}-1)} \iint_{|x - y|_\infty\leq 2L}	\dist_N(u_i(x),u_i(y)) \d x \d y. \label{eq:icontrol3}
\end{split}
\end{equation}
	
	For \ref{item:parallelu}, We have that by the triangle inequality, for each $i \in \set{0,1}$, $n \in \mathbb N$,
	\begin{equation*}
	\begin{split}
		&\sum_{Q \in \mathcal Q_{k_{n}} }\sum_{\substack{Q' \in \mathcal Q_{k_{n+1}} \\ Q' \subset Q} }\int_{Q'} \dist_N(\E_{k_{n}} (u_i)(x), \E_{k_{n+1}} (u_i)(x)) \d x \\
		&\;\leq \notag\sum_{Q \in \mathcal Q_{k_{n}} }\sum_{\substack{Q' \in \mathcal Q_{k_{n+1}} \\ Q' \subset Q}}\Bigg[\int_{Q'} \dist_N(\E_{k_{n}}(u_i)(x), u_i(x)) \d x + \int_{Q'} \dist_N(u_i(x),\E_{k_{n+1,i}} (u_i)(x))\d x \Bigg].
	\end{split}
	\end{equation*}
	We will sum over all the subcubes $Q' \subset Q$, noting that $Q = \bigcup \{Q' : Q' \in \mathcal Q_{k_{n+1}}, Q' \subset Q\}$. One can control further by
		\begin{equation}
		 \label{eq:asinI}
		 \begin{split}
		  &\sum_{Q \in \mathcal Q_{k_{n}} }\int_{Q} \dist_N(\E_{k_{n,1}}(u_i)(x),u_i(x))\d x +\sum_{Q \in \mathcal Q_{k_{n}} }\int_{Q} \dist_N( \E_{k_{n+1}} (u_i)(x) , u_i(x))\d x \\
		&\qquad \leq \int_{\mathbb R^d} \dist_N(\E_{k_{n}} (u_i)(x), u_i(x)) \d x  + \int_{\mathbb R^d} \dist_N( \E_{k_{n+1}}( u_i)(x), u_i(x)) \d x
		\end{split}
		\end{equation}
	which is controlled using \eqref{eq:alphaover2ell} and \eqref{eq:medianinequality}.
	We deduce that \ref{item:parallelu} is controlled by
	\begin{multline}\label{eq:icontrol1}
	 \Gamma\sum_{i \in \set{0,1}}\sum_{n \in \mathbb N\setminus \{0\}}(2^{-n} + 2^{-(n+1)})  + \sum_{i \in \set{0,1}}\sum_{Q \in \mathcal Q_{k_{0}}}\fint_{Q}\int_{Q} \dist_N(u_i(x),u_i(y))\d x \d y  %= 2(2 + 1)\Gamma
	 \\\leq 6\Gamma + \sum_{i \in \set{0,1}}\frac{1}{L^d}\iint_{\substack{\R^d \times \R^d \\|x - y|_\infty\leq 2L}}	\dist_N(u_i(x),u_i(y)) \d x \d y.
	\end{multline}

	Next, for  \ref{item:perpu}, we have for any two cubes $Q,Q' \in \mathcal Q_{k_{n}}$ adjacent, using in order constantness on cubes and the triangle inequality,
	\begin{align}
		\notag &2^{(1 - k_{n})(d - 1)}(2^{-k_{n}} - 2^{-k_{n+1}})L^d\fint_{Q}\fint_{Q'}\dist_N(\E_{k_{n}} (u_i)(x) , \E_{k_{n}} (u_i)(y))\d x \d y \\
		&\qquad \leq \label{eq:thefirsttemr}2^{(1 - k_{n})(d - 1)}(2^{-k_{n}} - 2^{-k_{n+1}})L^d\Bigg[\fint_{Q}\fint_{Q'}\dist_N(u_i(x),u_i(y))\d x \d y\\
		&\qquad\qquad  \label{eq:dhslkjvoduq}+ \fint_{Q}\fint_{Q'}\dist_N(u_i(x),\E_{k_{n}} (u_i)(y))\d x \d y \\
		&\qquad\qquad  \label{eq:dhslkjvoduqII}+\fint_{Q}\fint_{Q'}\dist_N(\E_{k_{n}} (u_i)(x),  u_i(y))\d x \d y \Bigg].
	\end{align}
	The two terms \eqref{eq:dhslkjvoduq}--\eqref{eq:dhslkjvoduqII} are treated noting that the prefactor can be estimated using
	\[
		2^{(1 - k_{n})(d - 1)}(2^{-k_{n}} - 2^{-k_{n+1}})L^d \leq 2^{(1 - k_{n})(d - 1)} 2^{-k_{n}}L^d = 2^{d-1}2^{-d k_{n}} = |Q|/2
	\]
	so that
	\begin{equation}
	\label{eq:qdissapear}
	\begin{split}
		\frac{|Q|}{2}\fint_{Q}\fint_{Q'} \dist_N(u_i(x),\E_{k_{n}}(u) (Q'))\d x \d y & \leq\frac{1}{2}\fint_{Q}\int_{Q'}\dist_N(u_i(x),\E_{k_{n}} (u_i)(Q'))\d x\d y \\
		&= \frac{1}{2}\int_{Q'}\dist_N(u_i(x),\E_{k_{n}} (u_i)(Q'))\d x
	\end{split}
	\end{equation}
	and this will be estimated below similarly to  \eqref{eq:asinI}, see \eqref{eq:twotermsatonece}.
	
	So, we focus on the term \eqref{eq:thefirsttemr}. Since $Q'$ is a neighbour of $Q$, we first observe $Q \times Q' \subset Q \times 3Q$ so that \eqref{eq:thefirsttemr} is controlled by
\begin{equation}
 \begin{split}
	&2^{(1 - k_{n})(d - 1)}(2^{-k_{n}} - 2^{-k_{n+1}})L^d\fint_{Q}\fint_{Q'}\dist_N(u_i(x),u_i(y))\d x \d y\\
	&\qquad \leq 2^{(1 - k_{n})(d - 1)}(2^{-k_{n}} - 2^{-k_{n+1,i}})L^d|Q|^{-2}\int_{Q}\int_{3Q}\dist_N(u_i(x),u_i(y))\d x \d y\\
	&\qquad \leq 2^{(1 - k_{n})(d - 1)} 2^d(2^{-k_{n}} - 2^{-k_{n+1}})(2^{1 -k_{n}})^{-d}L^{-d}\int_{Q}\int_{3Q}\dist_N(u_i(x),u_i(y))\d x \d y \\
	&\qquad \leq 3^d\cdot 2^{k_{n}}(2^{-k_{n}} - 2^{-k_{n+1}})L^{-d}\int_{Q}\fint_{|h|_\infty \leq 3\cdot 2^{1-k_{n}}L}\dist_N(u_i(x),u_i(x + h))\d h\d x
	\end{split}
	\label{eq:otherterm}
	\end{equation}
	by the change of variable $y = x + h$, $|h|_\infty \leq |x - y|_\infty \leq \diam(3Q) \leq 2\cdot 3\cdot 2^{- k_{n}}L$.
	So, the term \ref{item:perpu} is estimated in view of  \eqref{eq:qdissapear} and \eqref{eq:otherterm} by
	\begin{multline}\label{eq:twotermsatonece}
		2d\sum_{i \in \set{0,1}}\sum_{n \in \mathbb N}\int_{\R^d}\dist_N(u_i(x),\E_{k_{n}} (u_i)(x))\d x \\+2d 3^d \sum_{i \in \set{0,1}}\sum_{n \in \mathbb N}(2^{-k_{n}} - 2^{-k_{n+1}})2^{ k_{n}}\times\\\int_{\R^d}\fint_{|h|_\infty \leq 3\cdot 2^{1-k_{n}}L}\dist_N(u_i(x),u_i(x + h))\d h\d x
	\end{multline}
	where the $2d$ arise as the sum over the $2d$ neighbours of the cube $Q \in \mathcal Q_{k_{n}}$
	
	Using \eqref{eq:alphaover2ell} for the first term and \eqref{eq:distbtwcubes} for the second one when $n \geq 1$, we control \eqref{eq:twotermsatonece} by
	\begin{multline}\label{eq:treeterms}
		2d\sum_{i \in \set{0,1}}\sum_{n \in \mathbb N}\frac{\Gamma}{2^n} +
		2d \sum_{i \in \set{0,1}}\sum_{n \in \mathbb N_*}(2^{-k_{n}} - 2^{-k_{n+1}})2^{ k_{n}}\frac{\Gamma}{2^{k_{n-1}}} \\+ 2d \sum_{i \in \set{0,1}}\int_{\R^d}\fint_{|h|_\infty \leq 6L}\dist_N(u_i(x),u_i(x + h))\d h\d x.
	\end{multline}
	The first term of \eqref{eq:treeterms} is equal to $8d \Gamma$. The second one can be controlled by
\begin{equation}
\begin{split}
	2d \Gamma \sum_{i \in \set{0,1}} \sum_{n \in \mathbb{N}_*} \left( \frac{1}{2^{k_{n-1}}} - \frac{2^{ k_{n}}}{2^{k_{n-1}} 2^{k_{n+1}}} \right)&\leq 2d \Gamma \sum_{i \in \set{0,1}} \sum_{n \in \mathbb{N}_*} \left( \frac{1}{2^{ k_{n-1}}} - \frac{1}{2^{k_{n+1}}} \right) \\
	&\leq 4d\Gamma \sum_{i \in \set{0,1}} \frac{1}{2^{k_{0}}} + \frac{1}{2^{k_{1}}} \leq 3d\Gamma.
\end{split}
\end{equation}
	The last term of \eqref{eq:treeterms} is estimated by
	\[
		\frac{d}{L^d} \sum_{i \in \set{0,1}}\iint_{\substack{\R^d\times \R^d\\ |x- y|_\infty \leq 6L}}\dist_N(u_i(x),u_i(y))\d x\d y.
	\]
	We deduce that \eqref{eq:treeterms} is estimated by 
	\begin{equation}\label{eq:tocontrolperpu}
		(8d + 3d)\Gamma + \frac{d}{L^d} \sum_{i \in \set{0,1}}\iint_{\substack{\R^d\times \R^d\\ |x- y|_{\infty} \leq 6L}}\dist_N(u_i(x),u_i(y))\d x\d y
	\end{equation}
	so that in turn \ref{item:perpu} is controlled by it.
	
	We conclude that \eqref{eq:ineq_smallllemma_varconnect2map} is controlled by \ref{item:umeetsv}, \ref{item:parallelu} and \ref{item:perpu} which are respectively controlled by \eqref{eq:icontrol3}, \eqref{eq:icontrol1} and  \eqref{eq:tocontrolperpu}
	\begin{equation*}
		 \big(6 + 	11d\big)\Gamma +  \int_{\R^d} \dist_N(u_0(x),u_1(x)) \d x + \frac{C_d}{L^d}\sum_{i \in \set{0,1}}\iint_{|x - y|_{\infty}\leq 6L}	\dist_N(u_i(x),u_i(y)) \d x \d y
	\end{equation*}
	by the choice of $\Gamma$ in \eqref{eq:choiceofalpha},
	this conclude the proof of the estimate \eqref{eq:ineq_smallllemma_varconnect2map} of the Proposition \ref{prop:smalllemma_varconnect2map} but with \guillemet{$|x - y|_\infty \leq 6L$}; Lemma \ref{lemma:scaling} below allows us to write \guillemet{$|x - y| \leq L$} at the price of increasing the constant $C_d$.
\end{Proof}

Inspection of the proof of Proposition \ref{prop:smalllemma_varconnect2map} shows that one could have chosen if \(N\) was convex
\[
	\E_k (u_i)(Q) = \fint_Q u_i.
\]
However our present choice of  $\E_k (u_i)(Q)$ implies that the essential images $\mathrm{Im}(U) \subset \mathrm{Im}(u_0) \cup \mathrm{Im}(u_1) \subset N$.

\begin{lemma}\label{lemma:scaling} Let $u : \R^d \to N$. 
	If $0 <\ell \leq L$,
	\begin{equation}\label{eq:scaling_lemma}
		\frac{1}{L^{d+1}}\iint_{\substack{\R^d \times \R^d \\ |x - y| \leq L}}\dist_{N}(u(x),u(y)) \d x \d y
		\leq\frac{2^{d + 1}}{\ell^{d + 1}}\iint_{\substack{\R^d \times \R^d \\ |x - y| \leq \ell}}\dist_{N}(u(x),u(y)) \d x \d y.
	\end{equation}
\end{lemma}
\begin{Proof}{Lemma}{lemma:scaling}
We have, by the triangle inequality and a change of variable
	\begin{equation}
\begin{split}
& \iint_{\substack{\R^d \times \R^d \\ |x - y| \leq L}}\dist_{N}(u(x),u(y)) \d x \d y \\
		&\qquad \leq \iint_{\substack{\R^d \times \R^d \\ |x - y| \leq L}}\dist_{N}(u(x),u(\tfrac{x+y}{2})) + \dist_{N}(u(\tfrac{x+y}{2}),u(y)) \d x \d y  \\
		&\qquad \leq 2^{1+d}\iint_{\substack{\R^d \times \R^d \\ |x - z| \leq L/2}}\dist_{N}(u(x),u(z)) \d x \d z.
\end{split}
\end{equation}
	By mathematical induction, for each $n \in \mathbb N$, 
	\begin{equation*}
		\iint_{\substack{\R^d \times \R^d \\ |x - y| \leq L}}\dist_{N}(u(x),u(y)) \d x \d y \leq \\
		2^{n(d +1)}\iint_{\substack{\R^d \times \R^d \\ |x - y| \leq L/2^n}}\dist_{N}(u(x),u(y)) \d x \d y.
	\end{equation*}
	Given $\ell \leq L$, there exists $n \in \mathbb N$ such that $L \leq 2^{n}\ell \leq 2L$. We deduce \eqref{eq:scaling_lemma}.
\end{Proof}

\section{Connecting two maps by a Sobolev map}\label{sec:connecting-two-maps-by-a-sobolev-map-and-corollaries}

Proposition \ref{prop:w11case} smooths the map obtained by proposition \ref{prop:smalllemma_varconnect2map}. So does Proposition \ref{prop:oncubes} in a localized way.

\begin{proposition}\label{prop:w11case}   
	If $u_0,u_1: \R^d \to N$ satisfy \eqref{eq:thelimitassumption} and \eqref{eq:intecondition},
	then, for each $L >0$,  there exists $U \in \dotW^{1,1}(\R^d \times (-L,L),N)$ such that  $\Tr_{\R^d \times \{-L\}}U = u_0$, $\Tr_{\R^d \times \{L\}}U = u_1$ and
	\begin{multline}
		\label{eq:w11case_ineq}\int_{\R^d \times (-L,L)} |\D U| \leq C_d\int_{\R^d} \dist_N(u_0(x),u_1(x)) \d x\\ +\frac{C_d}{L^d}\sum_{i \in \set{0,1}} \iint_{\substack{\R^d \times \R^d \\|x - y|\leq L}}	\dist_N(u_i(x),u_i(y)) \d x \d y
	\end{multline}
	where $C_d >0$ only depends on $d$.
\end{proposition}

\begin{remark} \label{rmk:scaling}
	In \cite[Theorem 1.8]{leoni2019traces}, in the framework of $N = \R^\nu, \dist_N(p,q) =|p - q|$, the same result is obtained by  other means. The authors assume the integrability condition \eqref{eq:intecondition}; the equivalent (see Proposition \ref{proposition_Lp_to_Linfty_translation}) condition than \eqref{eq:thelimitassumption}, for $i \in \set{0,1}$,
\[
\lim_{\epsilon \to 0}\sup_{|h| \leq \epsilon}\int_{\R^d}|u_i(x) - u_i(x +h)|  \d x = 0,
\]
is required there; corresponding estimates are obtained.
\end{remark}

\begin{proposition} \label{prop:oncubes} Let $Q = Q(0,1) \subset \R^{d+1}$ be the unit cube of radius one centered at the origin.
	If $u : \partial Q \to N$ is integrable and  equals to $\pointN \in N$ on each face except two opposed faces, then there exists a map $U \in \dotW^{1,1}(Q,N)$ of trace $\tr_{\partial Q}U = u$ and
	\[
	\int_{Q} |\D U| \leq C_d \int_{\partial Q}	\dist_N(u(x),\pointN) \d x.
	\]
	where $C_d >0$ is a constant depending only on $d$.
\end{proposition}

We write an interpolating tool  for the $\dotW^{1,1}$-extension, Proposition \ref{prop:w11case}.
\begin{lemma}\label{lemma:interpolation}
	
	If $P =\prod_{i = 1}^d[0,|\ell_i|] $ and $a,b \in N$, there exists $u \in \dotW^{1,1}(P\times \R, N)$ satisfying 
	\begin{enumerate}[label=(\roman*)]
		\item \label{item:geometryofthemap} If $t  < - \dist(x, \partial  P)/2$, $u(x,t) = a$ and  if $t  > \dist(x, \partial P)/2$, $u(x,t) = b$,
		\item \label{item:geometryofthemap_II} $\dist_N(a,u(x,t)) \leq  \dist_N(a,b)$ and  $\dist_N(b,u(x,t)) \leq  \dist_N(a,b)$.
	\end{enumerate}
	Moreover, 
	\begin{equation}\label{eq:usefulineqinlemmainterpolation}  	\int_{P \times \R}|\D u| \leq 4 \dist_N(a,b)\Leb^d(P)
	\end{equation}
	and
	\begin{equation}\label{eq:L1estimate}
		\sup_{ t\in \R}	\int_P \dist_N(a \one_{(-\infty,0)}(t) + b \one_{[0,+\infty)}(t), u(x,t))\d x \leq \dist_N(a,b)\Leb^d(P).
	\end{equation}
\end{lemma}

The proof follows the argument for \(\mathrm W^{1, p}\) with \(1 < p < 2\) \cite{jean2025book}.

\begin{Proof}{Lemma}{lemma:interpolation}
Since \(d\) is a geodesic distance, we can choose $\gamma \in C^{\infty}(\mathbb{R}, N)$ to be a mapping such that $\gamma(t) = a$ if $t \leq -1$, $\gamma(t) = b$ if $t\geq 1$ and $|\gamma'(t)| \leq \dist_N(a, b)$ if $-1 \leq t \leq 1$, and $\gamma(t) = b$ if $t \geq 1$. We define $u: P\times \R \to N$ for every $(x,t) \in P\times \mathbb{R}$ by
	\[
	u(t, x) = \gamma \left( \frac{2t}{\dist(x, \partial P)} \right).
	\]
	By definition we have \ref{item:geometryofthemap} and \ref{item:geometryofthemap_II}.
	Then, we record that $u \in \mathrm W^{1,1}_{\text{loc}}(P\times \mathbb{R}, N)$, with for  $(t, x) \in \mathbb{R} \times P$:
	\[
	|\D u(x,t)| \leq  \left|\gamma' \left( \frac{2t}{\dist(x, \partial P)} \right)\right |\left( \frac{|t|}{\dist(x, \partial P)^2} + \frac{1}{\dist(x, \partial P)} \right).
	\]
	Defining the set
	\(
	\Sigma = \left\{ (x,t) \in P \times \R: 2|t| \leq \dist(x, \partial P)\right\},
	\)
	we integrate and estimate
	\begin{equation*}
	\begin{split}
		\int_{P \times \R} |\D u|
		&\leq  \dist_N(a, b) \iint_{\Sigma} \frac{|t|}{\dist(x, \partial P)^2} + \frac{1}{\dist(x, \partial P)}  \d x\d t \\
		&\leq 2 \dist_N(a, b) \iint_{\Sigma} \frac{1}{\dist(x, \partial P)} \d x\d t = 4 \dist_N(a, b) \int_{P} \d x = 4 \prod_{i = 1}^d|\ell_i| \dist_N(a, b).
		\end{split}
	\end{equation*}
	which establish \eqref{eq:usefulineqinlemmainterpolation}.	
	For \eqref{eq:L1estimate}, we assume $t > 0$
	\begin{equation*}
	 \begin{split}
		&\int_{P} \dist_N(a \one_{[-\infty,0)}(t) + b \one_{[0,+\infty)}(t),u(x,t))  \d x\\
		&\qquad \leq \int_{\Sigma(t)}\dist\left(b, \gamma \left( \frac{2t}{\dist(x, \partial P)} \right) \right) \d x \leq \dist_N(a,b) \int_{\Sigma(t)} 1 -\frac{2t}{\dist(x, \partial P)}  \d x \\
		&\qquad \leq \dist_N(a,b) \Leb^d(\Sigma(t)\cap P)
		\leq  \dist_N(a,b) \prod_{i = 1}^d|\ell_i|
	\end{split}
	\end{equation*}
	and similarly for $t \leq 0$.
\end{Proof}

\begin{lemma}\label{lemma:interpolationonesided}
	
	%Let $N$ be a connected Riemannian submanifold of $\R^\nu$, $\nu \in \mathbb N$. 
	%	Denote $Q(0,\rho) \subset \R^d$ a cube centered at the origin and of radius $\rho>0$. 
	%Fix $d \in \mathbb N\setminus \set{0}$. 
	If $P =\prod_{i = 1}^d[0,|\ell_i|] $ and $a,b \in N$, there exists $u \in \dotW^{1,1}(P\times (0,\infty), N)$ satisfying 
	\begin{enumerate}[label=(\roman*)]
		\item 
		If $t  > \dist(x, \partial P)/2$, $u(x,t) = b$.
		\item 
		$\dist_N(a,u(x,t)) \leq  \dist_N(a,b)$ and  $\dist_N(b,u(x,t)) \leq  \dist_N(a,b)$.
			\item $\tr_{P \times \{0\}}u = a$.
	\end{enumerate}
	Moreover, 
	\begin{equation}\label{eq:usefulineqinlemmainterpolationoneside}  	\int_{P \times \R}|\D u|  \leq 8 \dist_N(a,b)\Leb^d(P)
	\end{equation}
	and
	\begin{equation}\label{eq:L1estimateoneside}
		\sup_{ t>0}	\int_P \dist_N(b , u(x,t))\d x \leq 2\dist_N(a,b)\Leb^d(P).
	\end{equation}
\end{lemma}
\begin{Proof}{Lemma}{lemma:interpolationonesided}
	Take $\gamma \in C^\infty(\R,\mathcal N)$ to be a mapping such that $\gamma(0) = a$, $|\gamma'(t)| \leq 2 \dist_N(b,a)$ and $\gamma(t) = b, t\geq 1$. We proceed as in the proof of Lemma \ref{lemma:interpolation}.
\end{Proof}

\begin{Proof}{Proposition}{prop:w11case}
	By Proposition \ref{prop:smalllemma_varconnect2map}, there exists a map $U_{\ast}: \R^d \times (-L,L) \to N$ constant on a countable family of rectangular cuboïds. 
	
	By Lemma \ref{lemma:interpolation} applied to each  faces \emph{i.e.} element of
	\begin{multline*}
		\bigl\{ \partial{Q \times I_{i,n}} \cap \partial{Q \times I_{i',n'}}: \\ i,i' \in \set{0,1}, n,n' \in \mathbb N, Q \in \mathcal Q_{k_{n}}, Q' \in \mathcal Q_{k_{n'}} \text{ s.t }  \partial{Q \times I_{n, i}} \neq \partial{Q \times I_{n', i'}}\bigr\}
	\end{multline*}
	 we obtain a map $U : \R^d \to N$. First, by Lemma \ref{lemma:interpolation} \ref{item:geometryofthemap} this is well defined : the support of the modifications is disjoint. By Lemma \ref{lemma:interpolation}\eqref{eq:usefulineqinlemmainterpolation}  we get
	 \[
	 	\int_{\R^d \times (-L,L)}|\D U| \d \Leb^{d + 1} \leq 4 	\int_{\mathrm{J}_{U_\ast} \cap \R^d\times (-L,L)} \dist_N(\Tr_{\mathrm{J}_{U_\ast}}^+{U_\ast},\Tr_{\mathrm{J}_{U_\ast}}^-{U_\ast})\d \mathcal H^{d}.
	 \]
	By the choice of $U_*$ and in particular Proposition \ref{prop:w11case}\eqref{eq:ineq_smallllemma_varconnect2map} we get the estimate \eqref{eq:w11case_ineq}.
	
	The estimate \eqref{eq:L1estimate} will only be used to the faces perpendicular $\mathcal F^{\perp}$ to $\R^d\times\{0\}$.  
	If for $i\in \set{0,1}, n\in\mathbb N$, $I_{i,n} = (a,b)$, we set $t_{i,n} = (a + b)/2$. The choice of $t \in \R$ ensures that the plane $\R^d\times\{t\}$ does not intersect the support of the modification of the faces parallel to $\R^d \times \set{0}$. The choice of $t$ allows us to ensure that $\R^d \times \{t\}$ only crosses points in $\{x \in \R^{d} \times (-L,L) \,:\, U(x) \neq \tilde U(x)\}$ arising from modification of the perpendicular faces by Lemma \ref{lemma:interpolation}. The estimate \eqref{eq:L1estimate} implies
	\begin{multline}\label{eq:thatchoiceoft}
		\int_{\{x \in \R^{d}  \,:\, U(x) \neq \tilde U(x)\}}\dist_N(\tr_{\R^d\times \set{t_{i,n}}}U(x), \tr_{\R^d\times \set{t_{i,n}}}U_*(x)) \d x \\\leq \int_{\mathrm{J}_{U_\ast}}\dist_N(\tr^+_{\mathrm{J}}{U_\ast},\tr^-_{\mathrm{J}}{U_\ast})\d \Hau^d.
	\end{multline}
	By the triangle inequality, \eqref{eq:thatchoiceoft} and Proposition \ref{prop:smalllemma_varconnect2map}\eqref{eq:tracelimit}, for all $n \in \mathbb N$, $i \in \set{0,1}$,
	\begin{equation}
		\int_{\R^d}\dist_{N}(u_i(x),\tr_{\R^d\times \set{t_{i,n}}}U(x))\d x \leq 2^{-n}C(d,u_0,u_1,L)
	\end{equation}
	where $C(d,u_0,u_1,L)>0$ only depends on $d,u_0,u_1$ and $L$. By continuity of the trace under translations, since $t_{i,n} \to (-1)^{i+1}L$ when $n \to + \infty$, we conclude that $\tr_{\R^d \times \set{-L}}U = u_0$ and
	$\tr_{\R^d \times \set{L}}U = u_1$.
\end{Proof}

The proof of Proposition~\ref{prop:oncubes} is similar to Proposition \ref{prop:w11case} and uses the Lemma \ref{lemma:interpolationonesided} to handle the values on the map on the boundary of the cube.

\begin{Proof}{Proposition}{prop:oncubes}  We assume that $u$ is not constant on the faces $\partial Q \setminus (\R^d\times \{-1,1\})$. We define $v_0 \in \mathrm{L^1}(\R^{d},N)$ by $v_0(x) = u (x)$, $x \in \partial Q \setminus (\R^d\times \{-1\})$ and $v_0(x) = p$ elsewhere. We define $v_1 \in \mathrm{L^1}(\R^{d},N)$ by $v_1(x) = u (x)$, $x \in \partial Q \setminus (\R^d\times \{1\})$ and $v_1(x) = p$ elsewhere. By Proposition \ref{prop:smalllemma_varconnect2map} there exists $V \in \dot{\mathrm{BV}}(\R^d \times (-1,1),N)$ such that $\Tr_{\R^d \times \{-1\}}V = v_0$ and $ \tr_{\R^d \times \{1\}} = v_1$ satisfying $V |_{\R^d\times (-1,1) \setminus Q} = p$ and by \eqref{eq:ineq_smallllemma_varconnect2map} and the triangle inequality
	\[
	\int_{\R^d \times (-1,1)} \dist_N(\tr_{\mathrm J_V}^+V, \tr_{\mathrm J_V}^-V)\d \mathcal H^{d}\leq C_d \int_{\partial Q} \dist_N(u(x),\pointN) \d x
	\]
	Considering $V|_{Q}$, we then argue as in the proof of Proposition \ref{prop:w11case} using Lemma \ref{lemma:interpolationonesided} instead of Lemma \ref{lemma:interpolation} for faces of cuboïds adjacent to $\partial Q \setminus (\R^{d}\times \{-1,+1\})$. 
\end{Proof}

\section{Extension to a halfspace and to a manifold}
\label{sec:extension-from-manifold-domains}

Based on Proposition \ref{prop:oncubes}, we prove the extension part of Theorem \ref{thm:yetanothertraceineqRd}.

\begin{proposition}\label{prop:w11casehalfspace}
If $u : \R^d \to N$ is measurable and the right-hand side of \eqref{eq:propextthmccl} is finite,
 	then there exists $U \in \dotW^{1,1}(\R^d \times (0,\infty),N)$ such that $\tr_{\R^d \times \set{0}}U = u$ and
 	\begin{equation}\label{eq:propextthmccl}
 		\int_{\R^d \times (0,\infty)} |\D U|
 		\leq C_d\liminf_{R \to +\infty}
 		\iint\limits_{\substack{\R^d \times \R^d \\ |x - y| \leq R}} \frac{\dist_{N}(u(x), u(y))}{\mathcal L^d(\B^d(0,R))} \d x \d y,
 	\end{equation}
	where $C_d >0$ only depends on $d$.
\end{proposition}
\begin{proof}
Since \(u\) is integrable by Proposition \ref{prop:ballBBMII}, there exists \(b\) such that
\[
 \int_{\R^d} \dist_N (u (x), b) \d x \le C_d
 \liminf_{R \to \infty} \iint\limits_{\substack{\R^d \times \R^d \\ |x - y| \leq R}} \frac{\dist_{N}(u(x), u(y))}{\mathcal L^d(\B^d(0,R))} \d x \d y < \infty,
\]
We set \(u_0 = u\) and \(u_1 = b\) and we apply Proposition~\ref{prop:oncubes}.
\end{proof}

We finally construct the extension in Theorem~\ref{thm:W11extmanifold}.
\begin{proposition}\label{prop:W11extmanifold}
	Let $M$ be a  Riemannian manifold  with compact boundary $\partial M$. If $u : \partial M \to N$ is integrable, there exists $U \in \dotW^{1,1}(M,N)$ satisfying $\tr_{\partial M}U = u$ and
	\[
	\int_{M} |\D U|\leq C_M \iint_{\partial M \times \partial M}	\dist_N(u(x), u (y)) \d x
	\]
	where $C_M>0$ only depends on $M$.
\end{proposition}

The integration on $M$ and on $\partial M$ are performed with respect to the measure induced by their Riemannian metric or, equivalently, by the Hausdorff measure, and can be computed in local charts thanks to the introduction of the appropriate Jacobian.

\begin{Proof}{Proposition}{prop:W11extmanifold} Let $m$ denote the dimension of the manifold $M$. By smoothness and compactness of the boundary $\partial M$, we realize $\partial M$ as a simplicial complex \cite[Theorem 7]{whitehead1940onC1} $S_{\partial M} = \bigcup_{n = 0}^{m - 1} S_n$ of dimension $m - 1$ with $S_n$ containing the cells of dimension $n$; a cell $C \in S_{n}$ is bi-Lipschitz homeomorphic to a cube $Q(0,1) \subset \R^n$.
	Let $V = \{x \in M \cup \partial M: \dist_M(x,\partial M) \leq \rho\}$ for $\rho > 0$ sufficiently small so that there exists a diffeomorphism $\Psi:V \to \partial M \times [0,1]$. We realize $V$ as a  finite simplicial complex $S_M$ of dimension $m$, such that $S_{\partial M}$ is the boundary of $S_M$ meaning that for a $k$-dimension cell $C$ in $S_M$, $C \cap \partial M$ is a $(k-1)$-dimensional cell in $S_{k - 1}$ ; this can be done by taking $C \in S_{\partial M}$ and defining $C' = C \times [0,1] \subset \partial M \times [0,1]$ and  define the cell in $V$ by $\Psi^{-1}(C')$.	
	Fix $\pointN \in N$ so that
	\[
	\int_{\partial M} \dist_N(u(x),b) \d x
		\le \frac{1}{
		\Hau^d
		(\partial M)}\iint_{\partial M \times \partial M} \dist_N(u(x), u(y)) \d x \d y
	\]
	If $C \in S_{m - 1}$, on $\partial C' \setminus C \times \set{0}$ we extend the domain of definition of $u$ by setting $u \circ \Psi^{-1} = p$. Therefore $C'$ is a cell such that $\tr_{\partial C'}u \in \mathrm L^1(\partial C', N)$ and if two cells $C'$ are adjacent they share the same constant on their boundary. Using the bi-Lipschitz equivalence of cells with cubes, the situation reduces to extend a map $u \in \mathrm L^1(\partial Q,N)$, $Q = Q(0,1) \subset \R^m$ equals to $p$ on each face exept one to a map $U \in \dotW^{1,1}(Q,N)$ of trace $\tr_{\partial Q}U = u$, which is done by Proposition \ref{prop:oncubes} applied to each cube.
	We set $U = p$ on $M \setminus V$. Since the same trace is shared on the faces of the $m$-dimensional cells in $V$, we obtain, by integration by part (\textit{cfr} \cite[Theorem 18.1(ii)]{leoni2017sobolev}\cite{jean2025book}), a map $U \in \dotW^{1,1}(M,N)$ such that $\tr_{\partial M} U = u$ such that
	\begin{equation*}
	\begin{split}
		\int_{M}|\D U|= \int_{V} |\D U| &= \sum_{\substack{C' = C \times [0,1] \\ C \in S_{m-1}}}\int_{C'}|\D U|
		\\
		&\leq C_M \sum_{\substack{C' = C \times [0,1] \\ C \in S_{m-1}}}\int_{\partial C'}\dist_N(u(x),\pointN) \\
		&\leq C_M \sum_{C \in S_{m-1}}\int_{C}\dist_N(u(x),\pointN) = C_M \int_{\partial M}\dist_N(u(x),\pointN)
	\end{split}
	\end{equation*}
	where $C_M>0$ depends on the  Lipschitz constant of $\Psi$ and absorbs the constant arising from Proposition \ref{prop:oncubes},  which concludes the proof.
\end{Proof}

% \bibliographystyle{amsabbrv} 
%\bibliographystyle{alpha} % We choose the "plain" reference style
% \bibliography{ref.bib} % Entries are in the ref.bib file

\providecommand{\bysame}{\leavevmode\hbox to3em{\hrulefill}\thinspace}
\providecommand{\MR}{\relax\ifhmode\unskip\space\fi MR }
\providecommand{\MRhref}[2]{%
  \href{http://www.ams.org/mathscinet-getitem?mr=#1}{#2}
}
\providecommand{\href}[2]{#2}

\end{document}